\documentclass[12pt]{amsart}
\usepackage{amssymb,latexsym}
\usepackage{enumerate}

\makeatletter
\@namedef{subjclassname@2010}{%
  \textup{2010} Mathematics Subject Classification}
\makeatother

\newtheorem{thm}{Theorem}[section]
\newtheorem{cor}[thm]{Corollary}
\newtheorem{lem}[thm]{Lemma}
\newtheorem{prop}[thm]{Proposition}

\theoremstyle{definition}
\newtheorem{defin}[thm]{Definition}

\numberwithin{equation}{section}

\usepackage{color}
\newcommand{\red}{\color{black}}

\newcommand{\E}{\mathbb{E}}
\newcommand{\Prob}{\mathbb{P}}

\newcommand{\Pfs}[1][]{\ensuremath{\mathbb{P}_{#1}\text{-a.s.}}}

\newcommand{\N}{\mathbb{N}}

\newcommand{\R}{\mathbb{R}}

\newcommand{\C}{\mathbb{C}}

\newcommand{\supp}{\mathrm{supp}\,  }

\renewcommand{\epsilon}{\varepsilon}
\renewcommand{\rho}{\varrho}

\newcommand{\abs}[1]{\ensuremath{\left| {#1} \right|}}

\newcommand{\eqdist}{\stackrel{d}{=}}

\newcommand{\Rd}{\R^d}
\newcommand{\Rdnn}{{\R^d_\ge}}
\newcommand{\Rp}{\R_>}
\newcommand{\Rnn}{\R_\ge}
\newcommand{\Cbf}[2][]{\mathcal{C}_b^{#1}\left( #2 \right)}
\newcommand{\condC}{(C)}
\newcommand{\Sp}{\mathbb{S}_\ge}
\newcommand{\Sd}{\mathbb{S}}
\newcommand{\Mset}{\mathcal{M}_+}
\newcommand{\interior}[1]{\mathrm{int}({#1})}
\newcommand{\mT}{\mathbf{T}}
\newcommand{\mPi}{\matrix{\Pi}}

\newcommand{\mA}{\matrix{A}}
\renewcommand{\matrix}[1]{\mathbf{#1}}

\newcommand{\as}{\cdot}

\newcommand{\No}{\mathbb{N}_0}
\newcommand{\ma}{\matrix{a}}
\newcommand{\closure}[1]{\overline{#1}}

\renewcommand{\P}[2][]{\ensuremath{\mathbb{P}_{#1} \left( {#2} \right)}}

\begin{document}

%%%%% To ease editing, for IMPAN journals add:

 \baselineskip=17pt

%%%%%%%%%%%

%% In the running head, replace first names by initials 
%% and give an abbreviation of the title.

\title{A Note on Kesten's Choquet-Deny Lemma}

\author{Sebastian Mentemeier}
\address{Institute for Mathematical Statistics\\ University of M\"unster\\
Einsteinstra\ss e 62\\ 48149 M\"unster, Germany}
\email{mentemeier@uni-muenster.de}

\date{}

\begin{abstract}
Let $d >1$ and $(\mA_n)_{n \in \N}$ be a sequence of
independent identically distributed random $d\times d$ matrices with nonnegative entries. This induces a Markov chain $M_n = \mA_n M_{n-1}$ on the cone $\Rdnn
\setminus \{0\} = \Sp \times \Rp$. We study harmonic functions of this Markov
chain. In particular, it is shown that all bounded harmonic functions in
$\Cbf{\Sp} \otimes \Cbf{\Rp}$ are constant. The idea of the proof is originally
due to Kesten [\emph{Renewal theory for functionals of a Markov chain with general
state space}. Ann. Prob. 2 (1974), 355 -- 386], but is considerably
shortened here. {\red A similar result for invertible matrices is given as well}.
\end{abstract}



\maketitle

\section{Introduction}

Let $d >1$. Write $\Rdnn = [0, \infty)^d$  for the cone of $d$-vectors with nonnegative entries  and $\Sp := \{ x
\in \Rdnn \ : \ \abs{x}=1 \}$ for its intersection with the {\red unit sphere $\Sd$}, where $\abs{\cdot}$ is the euclidean norm on $\Rd$.
{\red A matrix $\ma \in  \Mset:=M(d \times d, \Rnn)$ is called \emph{allowable}, if it has no zero line or column. } 
Any {\red allowable} matrix leaves
$V:= \Rdnn \setminus \{0\}$ invariant and one can define its action on $\Sp$ by 
$$ \ma \as x := \frac{\ma x}{\abs{\ma x}}, \qquad x \in \Sp.$$

If $\mu$ is a probability distribution on {\red allowable matrices in $\Mset$
% satisfying
%\begin{equation}
%\label{nozerocolumn} \mu\left( \{ \ma \ : \ \ma \text{ has a zero column} \}
%\right) = 0,
%\end{equation}
then $V$ is $\mu$-a.s. invariant. Let $(\mA_n)_{n \in \N}$ be a sequence of independent identically distributed
(iid) random matrices with law $\mu$, then $M_n = \mA_n M_{n-1}$ defines a Markov chain on $V$. 

}

The aim of this note is to study the bounded harmonic functions of $(M_n)_{n \in \No}$ under some
additional equicontinuity condition {\red on the functions.}
 Besides being of interest in its own right,
the absence of nontrivial bounded harmonic functions appears prominently in the proof of Kesten's renewal theorem \cite{Kesten1974} and has been recently used in \cite{Mentemeier2013} to determine the set of fixed
points of the multivariate distributional equation
\begin{equation} Y \eqdist \mT_1 Y_1 + \dots \mT_N Y_N, \label{SFPE}
\end{equation} where $N \ge 2$ is fixed, $Y, Y_1,
\dots, Y_N$ are iid $\Rdnn$-valued random variables and independent of the random
matrices $(\mT_1, \dots, \mT_N) \in \Mset^N$.  

{\red The idea of proof is based on the Choquet-Deny lemma used by Kesten in the proof of his renewal theorem \cite[Lemma 1]{Kesten1974}. By restricting to the more specialized setting of Markov chains generated by the action of nonnegative matrices  and using recent results on products of random matrices from  \cite{BDGM2014}, this proof can be considerably shortened and, as we hope, thereby made more illuminating.
}

\section{Statement of Results}

{\red The Markov chain $(M_n)_{n \in \No}$ will be studied in the decomposition
$M_n = e^{S_n} X_n$ with $X_n \in \Sp$ and $S_n \in \R$. Note that up to an exponential transform, this corresponds to the decomposition $V=\Sp \times \Rp$. One easily deduces that
$$ X_n = \mA_n \as X_{n-1}, \qquad S_{n} - S_{n-1} = \log \abs{\mA_n
X_{n-1}},$$
hence $(X_n, S_n)_{n \in \No}$ is a Markov chain on $\Sp \times \R$ that carries the additional structure of a Markov random walk.}

Writing $\interior{A}$ for the topological interior of a set $A$, recall
that by the Perron-Frobenius Theorem, any $\ma \in \interior{\Mset}$ possesses a
unique largest eigenvalue $\lambda_\ma \in \Rp$ with corresponding normalized eigenvector
$w_\ma \in \interior{\Sp}$. Given a subsemigroup $\Gamma \subset \Mset$, write
$$ \Lambda(\Gamma) = \{ \log \lambda_\ma \ : \ \ma \in \Gamma \cap
\interior{\Mset} \}$$
for the set of eigenvalues and
$$ W(\Gamma) = \closure{ \{w_\ma \ : \ma \in \Gamma \cap \interior{\Mset}
\} }$$
for the closure of the set of normalized eigenvectors.

\begin{defin}
A subsemigroup $\Gamma \subset \Mset$ 
is said to satisfy condition $\condC$, if
\begin{enumerate}[(1)]
 \item every $\ma \in \Gamma$ is allowable and
  \item \label{2} $\Lambda(\Gamma)$ generates a dense subgroup of $\R$. 
\end{enumerate}
\end{defin}

A sufficient condition for (2) is that $\Gamma$ consists of invertible matrices and that there is no proper $\Gamma$-invariant subspace $W$ with $W \cap \Rdnn \neq \{0\}$, see the discussion in \cite{BDGM2014}.

Denote by $[\supp \mu]$ the smallest
closed semigroup of $\Mset$ generated by $\supp \mu$ and write $\Cbf{E}$ for the
set of continuous bounded functions on the space $E$. Abbreviating $\mPi_n =\mA_n
\dots \mA_1$, define for each $x \in \Sp$ a probability measure $\Prob_{x}$
on the path space of $(X_n, S_n)_{n \in \No}$ by 
$$ \Prob_{x}\Bigl((X_0, S_0, \dots, X_n, S_n) \in B \Bigr) = \Prob
\Bigl((x,0, \dots, \mPi_n \as x, \log \abs{\mPi_n  x} ) \in B \Bigr)$$
for all $n \in \N$ and measurable $B$. The corresponding expectation symbol is denoted by $\E_x$.

\begin{thm}\label{maintheorem}
Let $[\supp \mu]$ satisfy $\condC$. Assume
that $L \in \Cbf{\Sp \times \R}$ satisfies 
\begin{enumerate}[(a)]
  \item \label{1} $L(x,s) = \E_x\,{L(X_1, s- S_1)}$ for all $(x,s) \in \Sp \times
  \R$, and
  \item \label{2} for all $z \in \interior{\Sp}$, 
  $$ \lim_{y \to z} \sup_{t \in \R} \abs{L(y,t) - L(z,t)}=0 .$$ 
  \end{enumerate}
Then $L$ is constant.
\end{thm}

It is interesting to observe that $(b)$ is an equicontinuity property for the family $(L(\cdot,t))_{t \in \R}$ of functions in $\C_b(\Sp)$. In fact, the Arzel\`{a}-Ascoli theorem is applicable and yields that for all $t \in \R$
$$ \lim_{s \to t} \sup_{y \in \R} \abs{L(y,s)-L(y,t)}=0.$$
Each pair of functions $f \in \Cbf{\Sp}$, $h \in \Cbf{\R}$ defines a composite function
$f \otimes h \in \Cbf{\Sp \times \R}$ by $(f \otimes h)\, (u,s):=f(u)h(s)$. 
Write $\Cbf{\Sp} \otimes \Cbf{\R}$ for the set of all finite linear combinations of such
functions (tensor product). Then the following corollary is obvious:

\begin{cor}
Let $[\supp \mu]$ satisfy condition $\condC$. If $L
\in \Cbf{\Sp} \otimes \Cbf{\R}$ is
harmonic for the Markov chain $(X_n, S_n)_{n \in \No}$, then $L$ is constant.
\end{cor}

The further organisation of the paper is as follows. At first, we repeat for
the readers convenience important implications of $\condC$, based on
\cite{BDGM2014}. Then we turn to the proof of the main theorem. It will be
assumed throughout that $d>1$ and that $[\supp \mu]$ satisfies $\condC$.
{\red Finally, we describe briefly how to extend the result to Markov chains on $\Rd$ generated by the action of invertible matrices.}

\section{Implications of Condition $\condC$}

Under each $\Prob_x$, $x \in \Sp$, $(X_n)_{n \in \No}$ constitutes a Markov
chain with transition operator $P : \Cbf{\Sp} \to \Cbf{\Sp}$ defined by
$$ Pf(y) = \int f(\ma \as y)\ \mu(d\ma) = \E f(\mA_1 \as y) , \qquad y \in \Sp.$$
Abbreviating $\Gamma = [ \supp \mu]$, we have the following result.

\begin{prop}[{\cite[Lemma 4.4]{BDGM2014}}]\label{prop:ip}
There is a unique
$P$- stationary probability measure $\nu$ on $\Sp$, and $\supp \nu = W(\Gamma)$. 
\end{prop}

Since $\Sp$ is compact, the uniqueness of $\nu$ implies the following ergodic
theorem (see \cite{Breiman1960})
\begin{equation} \label{eq:ergodic}
\lim_{n \to \infty} \frac{1}{n} \sum_{k=1}^n f(X_k) = \int f(y)\ \nu(dy) \qquad
\Pfs[x]
\end{equation}
for all $x \in \Sp$, $f \in \Cbf{\Sp}$.

Property (2) of $\condC$ implies the following ``weak'' aperiodicity property
of $(S_n)_{n \in \No}$, which is an adaption of condition $I.3$ in
\cite{Kesten1974}. As usual, $B_\epsilon(z) := \{y \in E : \abs{y-z} < \epsilon\}$ for $\epsilon >0$, $z \in E$.

\begin{lem}\label{lem:aperiodic}
There exists a sequence $(\zeta_i)_{i \in \N} \subset \R$ such that the group
generated by $(\zeta_i)_{i \in \N}$ is dense in $\R$ and such that for each
$\zeta_i$ there exists $z \in \interior{\Sp}$ with
the following properties:
\begin{enumerate}
  \item $\nu\bigl(B_\epsilon(z)\bigr) >0$ for all $\epsilon >0$.
  \label{recurrence_property}
\item
 For all $\delta >0$ there is $\epsilon_\delta >0$ such
that for all $\epsilon \in (0, \epsilon_\delta)$ there are $m \in \N$ and $\eta
>0$, such that for $B:=B_\epsilon(z)$:
\begin{align}\label{eq:aper}
\P[x]{ X_m \in B, \ \abs{S_{m} - \zeta_i} < \delta } \ge \eta \quad \text{ for
all } x \in B.
\end{align}
\end{enumerate}
\end{lem}

The first property together with \eqref{eq:ergodic} entails that $B$ is a
recurrent set for $(X_n)_{n \in \N}$. 
By a geometric trials argument (see e.g. \cite[Problem 5.10]{Breiman1968}), it
follows that for all  $\delta >0$ and sufficiently small $\epsilon>0$ there is
$m \in \N$ such that
\begin{equation} \label{eq:i.o.}
\P[x]{  \abs{X_n-z} < \epsilon, \abs{X_{n+m}-z} < \epsilon, \abs{S_n -
(S_{n+m} - \zeta_i)} < \delta  \text{ i.o.
%\footnote{i.o. stands for ``infinitely often'' and here refers to $N$.}
 } }=1
\end{equation}

We repeat the short
proof of Lemma \ref{lem:aperiodic} from \cite[Prop. 7.5]{BDGM2014}, for it
clarifies the importance of part (2) of $\condC$ and moreover, we want to strengthen
the result a bit.

\begin{proof}
By part (2) of $\condC$, the set $\Lambda(\Gamma)$ generates a dense
subgroup of $\R$, hence it contains a countable sequence $(\zeta_i)$ which still
generates a dense subgroup. Fix $\zeta_i$. Then $\zeta_i = \log \lambda_\ma$ for some $\ma \in
\Gamma \cap \interior{\Mset}$, set $$z := w_\ma \in W(\Gamma) \cap
\interior{\Sp}.$$ Referring  to Prop. \ref{prop:ip}, $z \in \supp \nu$, thus
\eqref{recurrence_property} follows. 

Now fix $\delta >0$. Then for all $\epsilon >0$ sufficiently small, since $w_\ma$ is a Perron-Frobenius eigenvector,
\begin{align*}
 \ma \as B_{\epsilon}(w_\ma) \subset B_{\epsilon/2}(w_\ma),& \\
 \abs{\log \lambda_\ma - \log \abs{\ma x}} < \delta/2 & \qquad
 \text{ for all } x \in B_{\epsilon}(w_\ma).
\end{align*}

Since $\ma \in [\supp
\mu]$, there is $m \in \N$ such that $\ma = \ma_m \dots \ma_1$, $\ma_j \in
\supp \mu$, $1 \le j \le m$, hence for all $\gamma>0$,
$$ \P{ \mA_n \cdots \mA_1 \in B_\gamma(\ma) } =
\eta_\gamma >0 .$$ If $\gamma >0$ is chosen sufficiently small, then for all
$\ma' \in B_\gamma(\ma)$, 
\begin{align*}
\ma' \as B_{\epsilon}(w_\ma) \subset B_{\epsilon}(w_\ma),& \\
\abs{\log \lambda_\ma - \log \abs{\ma' x}} < \delta & \qquad
 \text{ for all } x \in B_{\epsilon}(w_\ma).
\end{align*}
Consequently, for all $x
\in B_{\epsilon}(w_\ma)$,
\begin{equation*}
\P{ \abs{ \mPi_n \as x - w_\ma } < \epsilon, \
\abs{\log \abs{\mPi_n x} - \log \lambda_\ma} < \delta } \ge
\eta_\gamma > 0. 
\end{equation*}
Recalling the definition of $\Prob_x$, this gives \eqref{eq:aper}.
\end{proof}

\section{Proof of the Main Theorem}

Let $L \in \Cbf{\Sp \times \R}$. For a compactly supported function $h \in
\Cbf{\R}$ define $$ L_h(x,s) = \int L(x,s+r)\, h(r)\, dr .$$
If for each such $h$, $L_h$ is constant, then the same holds true for $L$ itself
-- this can be seen by choosing a sequence $h_n$ of probability densities, such
that $h_n(r)\, dr$ converges weakly towards the dirac measure in 0. 

\begin{lem}
Let $L \in \Cbf{\Sp \times \R}$ satisfy properties (a),(b) of Theorem
\ref{maintheorem}. Then for any compactly supported $h \in \Cbf{\R}$, $L_h$
still satisfies (a),(b) and moreover: \begin{enumerate}[(c)]
  \item For all $z \in \interior{\Sp}$,
$$
\lim_{y \to z} \lim_{\delta \downarrow 0} \sup_{\abs{t-t'}< \delta}
\abs{L_h(z,t)-L_h(y,t')} = 0.$$
\end{enumerate}
\end{lem}

\begin{proof}
That $(a)$ and $(b)$ persist to hold for $L_h$ is a simple consequence of
Fubini's theorem resp. Fatou's lemma.

In order to prove $(c)$, let $\abs{L} \le C$. Consider
\begin{align*}
& \lim_{\delta \to 0} \sup_{y \in \Sp} \sup_{\abs{t-t'}< \delta}
\abs{L_h(y,t) - L_h(y,t')} \\
= &  \lim_{\delta \to 0} \sup_{y \in \Sp} \sup_{\abs{t-t'}< \delta}
\abs{\int L(y,t'+r) h(r-(t-t')) dr - \int L(y,t'+r) h(r) dr } \\
\le &  \lim_{\delta \to 0} \sup_{\abs{t-t'}< \delta}
C \int \abs{h(r-(t-t')) - h(r)  } dr = 0,
\end{align*}
where the uniform continuity of $h$ was taken into account for the last line.
Combine this with $(b)$ to obtain for all $z \in \interior{\Sp}$,
\begin{align*}
& \lim_{y \to z} \lim_{\delta \downarrow 0} \sup_{\abs{t-t'}< \delta}
\abs{L_h(z,t)-L_h(y,t')} \\
\le & \lim_{y \to z} \lim_{\delta \downarrow 0} \sup_{\abs{t-t'}< \delta}
\Bigl[ \abs{L_h(z,t)-L_h(y,t)} + \abs{L_h(y,t) - L_h(y,t')} \Bigr] \\
\le & \lim_{y \to z} \sup_{t \in \R} \abs{L_h(z,t)-L_h(y,t)} + \lim_{\delta \to
0} \sup_{y \in \Sp} \sup_{\abs{t - t'} < \delta} \abs{L_h(y,t) - L_h(y,t')}
= 0 .
\end{align*}
\end{proof}

Consequently, in order to proof Theorem \ref{maintheorem}, we may w.l.o.g.
assume that $L$ satisfies properties $(a) - (c)$.

\begin{proof}[Proof of Theorem \ref{maintheorem}]
The burden of the proof is to show that for all the $\zeta_i$ of Lemma
\ref{lem:aperiodic},
\begin{equation}\label{thismustbeshown} L(x,s) = L(x,s+\zeta_i) \quad \text{
for all } (x,s) \in \Sp \times \R .
\end{equation}
If this holds true, then for any $\sigma=\sum_{i=1}^N c_i \zeta_i$ with $c_i \in
\No$, $N \in \N$ $$ L(x,s)=L(x, s+ \sigma ) \quad \text{ for all } (x,s) \in \Sp \times \R
.$$ But the set of $\sigma$'s is dense in $\R$, thus by the continuity of $L$,
$$ L(x,s)=L(x,0) \quad \text{
for all } (x,s) \in \Sp \times \R .$$
Hence $L(x,s)$ reduces to a function $\tilde{L}$ on $\Sp$, which is then
bounded harmonic for the ergodic (see
\eqref{eq:ergodic}) Markov chain $(X_n)_{n \in \No}$, thus $\tilde{L}$ is constant. 

Now we are going to prove \eqref{thismustbeshown}. Considering $(a)$, $L(X_n, s-
S_n)_{n \in \No}$ constitutes a bounded, hence a.s. convergent martingale under
each $\Prob_x$ with
\begin{equation}\label{pointwise convergence} L(x,s) = \E_x \lim_{n \to \infty}
L(X_n, s - S_n) \quad \text{
for all } (x,s) \in \Sp \times \R .
\end{equation}
Fix any $\zeta_i$ and the corresponding $z \in \interior{\Sp}$, defined in Lemma
\ref{lem:aperiodic}. Referring to $(c)$, for all $\xi >0$, there are $\delta,
\epsilon >0$ such that $$ \sup_{u,y \in B_\epsilon(z)} \sup_{\abs{t-t'}< \delta} \abs{L(u,t) -
L(y,t')} < \xi .$$
Combining this with \eqref{eq:i.o.}, we infer that for all $s \in \R$,
$$ \P[x]{ \abs{L(X_n, s -S_n) - L(X_{n+m}, s + \zeta_i - S_{n+m})} < \xi \text{
i.o.}} = 1 .$$
Hence for all $(x,s) \in \Sp \times \R$,
$$ \lim_{n \to \infty} L(X_n, s- S_n) = \lim_{n \to \infty} L(X_n, s + \zeta_i
-S_n) \qquad \Pfs[x] $$
and consequently, using \eqref{pointwise convergence}, it follows for all
$(x,s) \in \Sp \times \R$ $$ L(x,s) = \E_x \lim_{n \to \infty} L(X_n, s - S_n) =
\E_x \lim_{n \to \infty} L(X_n, s + \zeta_i - S_n) = L(x, s + \zeta_i) .$$
\end{proof}

\section{Invertible Matrices}

Let us finally mention that a result similar to Theorem \ref{maintheorem} holds for invertible matrices: In the following, let $\mu$ be a distribution on $GL(d, \R)$.

\begin{defin}
A subsemigroup $\Gamma \in GL(d,\R)$ is said to be \emph{irreducible-proximal (i-p)}, if
\begin{enumerate}
\item no finite union $W=\bigcup_{i=1}^n W_i$ of proper subspaces $\{0\} \subsetneq W_i \subsetneq \Rd$ satisfies $\Gamma W \subset W$ (irreducibility) and
\item there is $g \in \Gamma$ having a algebraically simple dominant eigenvalue $\lambda_g \in \R$
 such that $|{\lambda_g}| = \lim_{n \to \infty} \|{g^n}\|^{1/n}$ (proximality). 
\end{enumerate}
\end{defin}

This condition has been studied intensively by Guivarc'h and Le Page \cite{Guivarc'h2004,Guivarch2012,Guivarch2013}. Considering condition \emph{(i-p)}, the aperiodicity condition (2) follows from \cite[Proposition 2.5]{Guivarch2012}, which also gives the support of the measure $\nu$, corresponding to Proposition \ref{prop:ip}.Then following the lines of the proof of Theorem \ref{maintheorem}, one obtains the following:
\begin{thm}
Let $[\supp \mu]$ satisfy {(i-p)}. Assume
that $L \in \Cbf{\Sd \times \R}$ satisfies 
\begin{enumerate}[(a)]
  \item \label{1} $L(x,s) = \E_x\,{L(X_1, s- S_1)}$ for all $(x,s) \in \Sd \times
  \R$, and
  \item \label{2} for all $z \in \Sd$, 
  $$ \lim_{y \to z} \sup_{t \in \R} \abs{L(y,t) - L(z,t)}=0 .$$ 
  \end{enumerate}
Then $L$ is constant.
\end{thm}

\subsection*{Acknowledgements}
This research was supported by Deutsche Forschungsgemeinschaft (SFB 878).

\nocite{Breiman1960, Kesten1973, Kesten1974, Kluppelberg2003, Mentemeier2013}

\bibliographystyle{amsplain}   
\bibliography{literatur}

%\begin{thebibliography}{1}
%
%\bibitem{Breiman1960}
%{\sc L.~Breiman}, {\em The strong law of large numbers for a class of {M}arkov
%  chains}, Ann. Math. Statist., 31 (1960), pp.~801--803.
%
%\bibitem{Breiman1968}
%\leavevmode\vrule height 2pt depth -1.6pt width 23pt, {\em Probability},
%  Addison-Wesley, 1968.
%
%\bibitem{BDG2011}
%{\sc D.~{Buraczewski}, E.~{Damek}, and Y.~{Guivarc'h}}, {\em {On
%  multidimensional Mandelbrot's cascades}}, ArXiv e-prints,  (2011).
%\newblock available online at http://arxiv.org/abs/1109.1845.
%
%\bibitem{Kesten1973}
%{\sc H.~Kesten}, {\em {Random difference equations and renewal theory for
%  products of random matrices}}, Acta Math., 131 (1973), pp.~207--248.
%
%\bibitem{Kesten1974}
%\leavevmode\vrule height 2pt depth -1.6pt width 23pt, {\em {Renewal theory for
%  functionals of a Markov chain with general state space}}, Ann. Prob., 2
%  (1974), pp.~355--386.
%
%\bibitem{Kluppelberg2003}
%{\sc C.~Kl{\"u}ppelberg and S.~Pergamenchtchikov}, {\em Renewal theory for
%  functionals of a {M}arkov chain with compact state space}, Ann. Prob., 31
%  (2003), pp.~2270--2300.
%
%\bibitem{Mentemeier2013}
%{\sc S.~Mentemeier}, {\em On multivariate stochastic fixed point equations :
%  the smoothing transform and random difference equations}, PhD thesis,
%  University of Muenster, 2013.
%
%\end{thebibliography}

\end{document}